\documentclass[11pt,english]{article}

\usepackage[T1]{fontenc}
\usepackage[latin9]{inputenc}
\usepackage{geometry}
\usepackage{babel}
\usepackage{amsmath}
\usepackage{amsthm}
\usepackage{amssymb}
\usepackage{color}
\usepackage[unicode=true,pdfusetitle,
 bookmarks=true,bookmarksnumbered=false,bookmarksopen=false,
 breaklinks=false,pdfborder={0 0 0},pdfborderstyle={},backref=false,colorlinks=false]
 {hyperref}

\makeatletter

\usepackage[nameinlink,capitalise,noabbrev]{cleveref}

\theoremstyle{plain}
\newtheorem{thm}{Theorem}
\crefname{thm}{Theorem}{Theorems}
\theoremstyle{plain}
\newtheorem{lem}[thm]{Lemma}
\crefname{lem}{Lemma}{Lemmas}
\theoremstyle{plain}

\theoremstyle{plain}
\newtheorem*{claim*}{Claim}
\theoremstyle{definition}

\theoremstyle{plain}

\theoremstyle{plain}

\theoremstyle{definition}

\theoremstyle{definition}

\theoremstyle{plain}

\date{}

\usepackage{appendix}

\usepackage{bm}

\usepackage{bbm}

\usepackage[nameinlink,capitalise,noabbrev]{cleveref}
\crefname{appsec}{Appendix}{Appendices}
\crefname{enumi}{condition}{conditions}
\Crefname{enumi}{Condition}{Conditions}

\let\originalleft\left
\let\originalright\right
\renewcommand{\left}{\mathopen{}\mathclose\bgroup\originalleft}
\renewcommand{\right}{\aftergroup\egroup\originalright}
\usepackage{pgfplots}
\usetikzlibrary{pgfplots.groupplots}
\usepackage{verbatim}

\makeatletter
\renewcommand*{\UrlTildeSpecial}{%
  \do\~{%
    \mbox{%
      \fontfamily{ptm}\selectfont
      \textasciitilde
    }%
  }%
}%
\let\Url@force@Tilde\UrlTildeSpecial
\makeatother

\usepackage{tikz}
\tikzstyle{vertex}=[circle,draw=black,fill=black,inner sep=0,minimum size=0.2cm,text=white,font=\footnotesize]
\tikzset{every loop/.style={min distance=50,in=50,out=130,looseness=7}}

\usepackage[labelfont=bf,labelsep=period]{caption}

\global\long\def\mk#1{\textcolor{blue}{\textbf{[MK comments:} #1\textbf{]}}}
\def\mk#1{}

\makeatother

\begin{document}

\title{Ramsey graphs induce subgraphs of quadratically many sizes}

\author{Matthew Kwan \thanks{Department of Mathematics, ETH, 8092 Z\"urich. Email: \href{mailto:matthew.kwan@math.ethz.ch} {\nolinkurl{matthew.kwan@math.ethz.ch}}.}\and
Benny Sudakov\thanks{Department of Mathematics, ETH, 8092 Z\"urich, Switzerland. Email:
\href{mailto:benjamin.sudakov@math.ethz.ch} {\nolinkurl{benjamin.sudakov@math.ethz.ch}}.
Research supported in part by SNSF grant 200021-175573.}}

\maketitle
\global\long\def\RR{\mathbb{R}}
\global\long\def\QQ{\mathbb{Q}}
\global\long\def\E{\mathbb{E}}
\global\long\def\Var{\operatorname{Var}}
\global\long\def\CC{\mathbb{C}}
\global\long\def\NN{\mathbb{N}}
\global\long\def\ZZ{\mathbb{Z}}
\global\long\def\one{\mathbbm{1}}
\global\long\def\floor#1{\left\lfloor #1\right\rfloor }
\global\long\def\ceil#1{\left\lceil #1\right\rceil }
\global\long\def\cond{\,\middle|\,}

\global\long\def\mk#1{\textcolor{red}{\textbf{[MK comments:} #1\textbf{]}}}

\begin{abstract}
An $n$-vertex graph is called \emph{$C$-Ramsey} if it has no clique
or independent set of size $C\log n$. All known constructions of
Ramsey graphs involve randomness in an essential way, and there is
an ongoing line of research towards showing that in fact all Ramsey
graphs must obey certain ``richness'' properties characteristic
of random graphs. Motivated by an old problem of Erd\H os and McKay,
recently Narayanan, Sahasrabudhe and Tomon conjectured that for any
fixed $C$, every $n$-vertex $C$-Ramsey graph induces subgraphs
of $\Theta\left(n^{2}\right)$ different sizes. In this paper we prove
this conjecture.
\end{abstract}

\section{Introduction}

An induced subgraph of a graph is said to be \emph{homogeneous }if
it is a clique or independent set. A classical result in Ramsey theory,
proved in 1935 by Erd\H os and Szekeres \cite{ES35}, is that every
$n$-vertex graph has a homogeneous subgraph with at least $\frac{1}{2}\log_{2}n$
vertices. On the other hand, Erd\H os \cite{Erd47} famously used
the probabilistic method to prove that there exists an $n$-vertex graph with no
homogeneous subgraph on $2\log_{2}n$ vertices. Despite significant
effort (see for example \cite{FW81,BRSW12,Coh16,CZ16}), there are
no non-probabilistic constructions of graphs with comparably small
homogeneous sets. 

For some fixed $C$, say a graph is \emph{$C$-Ramsey} if it has no
homogeneous subgraph of size $C\log_{2}n$. It is widely believed
that $C$-Ramsey graphs must in some sense resemble random graphs,
and this belief has been supported by a number of theorems showing
that certain ``richness'' properties characteristic of random graphs
hold for all $C$-Ramsey graphs. The first result of this type was
due to Erd\H os and Szemer\'edi \cite{ES72}, who showed that $C$-Ramsey
graphs have density bounded away from 0 and 1. Further research has
focused on showing that certain statistics or substructures can take
many different values. Improving a result of Erd\H os and Hajnal
\cite{EH77}, Pr\"omel and R\"odl \cite{PR99} proved that for every
constant $C$ there is $c>0$ such that every $n$-vertex $C$-Ramsey
graph contains every possible graph on $c\log_{2}n$ vertices as an
induced subgraph. Shelah \cite{She98} proved that every $n$-vertex
$C$-Ramsey graph contains $2^{\Omega\left(n\right)}$ non-isomorphic
induced subgraphs. Fairly recently, answering a question of Erd\H os,
Faudree and S\'os \cite{Erd92,Erd97}, Bukh and Sudakov \cite{BS07}
showed that every $n$-vertex $C$-Ramsey graph has an induced subgraph
with $\Omega\left(\sqrt{n}\right)$ different degrees.

Two significant open problems in this area concern variation in the
numbers of edges and vertices in induced subgraphs. For a graph $G$,
let 
\begin{align*}
\Phi\left(G\right) & =\left\{ e\left(H\right):H\text{ is an induced subgraph of }G\right\} ,\\
\Psi\left(G\right) & =\left\{ \left(v\left(H\right),e\left(H\right)\right):H\text{ is an induced subgraph of }G\right\} .
\end{align*}

Erd\H os and McKay \cite{Erd92,Erd97} conjectured that for any $C$
there is $\delta>0$ such that for every $n$-vertex $C$-Ramsey graph
$G$, the set $\Phi\left(G\right)$ contains the interval $\left\{ 0,\dots,\delta n^{2}\right\} $.
Erd\H os, Faudree and S\'os \cite{Erd92,Erd97} conjectured that
for any fixed $C$ and any $n$-vertex $C$-Ramsey graph $G$, we
have $\left|\Psi\left(G\right)\right|=\Omega\left(n^{5/2}\right)$.
The best progress on the former conjecture is due to Alon, Krivelevich
and Sudakov \cite{AKS03}, who proved it with $n^{\delta}$ in place
of $\delta n^{2}$. The best progress on the latter conjecture is
due to Alon, Balogh, Kostochka and Samotij \cite{ABKS09} (improving
work of Alon and Kostochka \cite{AK09}), who proved it with $2.369$
in place of $5/2$. We also remark that strengthenings of both these
conjectures have been shown to hold for random graphs \cite{CFM92,AK09}.

Recently, Narayanan, Sahasrabudhe and Tomon \cite{NST16} proposed
a natural weakening of the aforementioned Erd\H os\textendash McKay
conjecture in the spirit of the Erd\H os\textendash Faudree\textendash S\'os
conjecture. Specifically, they conjectured that $\left|\Phi\left(G\right)\right|=\Omega\left(n^{2}\right)$
for every $n$-vertex $C$-Ramsey graph $G$, and proved the weaker
result that $\left|\Phi\left(G\right)\right|=n^{2-o\left(1\right)}$
(to be precise, they explain that their methods actually give a bound
of the form $n^{2}/e^{\Theta\left(\sqrt{\log n}\right)}$). In this
paper we prove Narayanan, Sahasrabudhe and Tomon's conjecture.
\begin{thm}
\label{conj:NST}For any fixed $C$, and any $n$-vertex $C$-Ramsey
graph $G$, we have $\left|\Phi\left(G\right)\right|=\Omega\left(n^{2}\right)$.
\end{thm}

We remark that the order of magnitude $n^{2}$ is best possible, because
$\Phi\left(G\right)\subseteq\left\{ 0,\dots,\binom{n}{2}\right\} $
for any $n$-vertex graph. Very loosely speaking, the general approach
of our proof is similar to the proof in \cite{NST16}, but we make
a number of simplifications and introduce some new ideas that we hope
will be useful for other problems.

\subsection{\label{subsec:definitions}Notation and basic definitions}

We use standard asymptotic notation throughout, and all asymptotics
are as $n\to\infty$ unless stated otherwise. Floor and ceiling symbols
will be systematically omitted where they are not crucial.

For two multisets $A$ and $B$, let $A\triangle B$ be the set of
elements which have different multiplicities in $A$ and $B$ (so
if $A$ and $B$ are ordinary sets, then $A\triangle B$ is the ordinary
symmetric difference $\left(A\backslash B\right)\cup\left(B\backslash A\right)$).
For a set $A$, we denote by $\binom{A}{2}$ the set of all unordered
pairs of elements of $A$.

We also use standard graph theoretic notation throughout. In particular,
in a graph, the density of a set of vertices $A$ is defined as $d\left(A\right)=e\left(A\right)/\binom{\left|A\right|}{2}$,
where $e\left(A\right)$ is the number of edges which are contained
inside $A$. Similarly, for any two sets $A$ and $B$, the density
between them is $d\left(A,B\right)=e\left(A,B\right)/\left|A\right|\left|B\right|$,
where $e\left(A,B\right)$ is the number of edges between $A$ and
$B$. For a vertex $v$ and a set of vertices $A$, we denote the
set of neighbours of $v$ in $A$ by $N_{A}\left(v\right)=N\left(v\right)\cap A$
and we denote the degree of $v$ into $A$ by $d_{A}\left(v\right)=\left|N_{A}\left(v\right)\right|$.

We also make some less standard graph theoretic definitions that will
be convenient for the proof. For a pair of vertices $\boldsymbol{v}=\left\{ v_{1},v_{2}\right\} $,
let $N\left(\boldsymbol{v}\right)$ (respectively $N_{U}\left(\boldsymbol{v}\right)$)
be the multiset union of $N\left(v_{1}\right)$ and $N\left(v_{2}\right)$
(respectively, of $N_{U}\left(v_{1}\right)$ and $N_{U}\left(v_{2}\right)$).
Let $d\left(v\right)=d\left(v_{1}\right)+d\left(v_{2}\right)$ (respectively
$d_{U}\left(\boldsymbol{v}\right)=d_{U}\left(v_{1}\right)+d_{U}\left(v_{2}\right)$)
be the size of $N\left(\boldsymbol{v}\right)$ (respectively, of $N_{U}\left(\boldsymbol{v}\right)$),
accounting for multiplicity.

\section{Ideas of the proof and previous work}\label{sec:outline}

As mentioned in the introduction, our proof builds on some ideas of Narayanan, Sahasrabudhe and Tomon in \cite{NST16}. This work in turn builds on the ideas of Bukh and Sudakov in \cite{BS07}. In this section we briefly outline the relevant ideas in both these papers, and discuss the new ideas in this paper.

In \cite{BS07}, Bukh and Sudakov proved that $n$-vertex Ramsey graphs have subgraphs with $\Omega(\sqrt n)$ distinct degrees. To do this, they introduced the notion of \emph{diversity}, as follows. Say an $n$-vertex graph is \emph{$\left(c,\gamma,\right)$-diverse} if for each vertex $x\in V$, we have $\left|N\left(x\right)\triangle N\left(y\right)\right|<cn$ for at most $n^{\gamma}$ vertices $y\in V$. Roughly speaking, this means the neighbourhoods of most pairs of vertices are very different. Bukh and Sudakov went on to prove that for any $C$ and $\gamma>0$, all $C$-Ramsey graphs have $(\Omega(1),\gamma)$-diverse induced subgraphs of linear size.

Now, in an $n$-vertex $(c,\gamma)$-diverse graph $G$, consider a random vertex subset $U$ obtained by including each vertex with some fixed probability $p$ independently. By the diversity assumption, for most pairs of vertices $u,v$
their degrees $d_{U}\left(u\right)$, $d_{U}\left(v\right)$ into
$U$ are not too strongly correlated, and the probability they are
exactly equal turns out to be $O\left(1/\sqrt{n}\right)$. (A simple
intuitive reason for this probability is that $d_{U}\left(u\right)-d_{U}\left(v\right)$
is approximately normally distributed with standard deviation $\Theta\left(\sqrt{n}\right)$).
One may then compute that the expected number of pairs of vertices with the same degree into $U$ is $O(n^{3/2}+n^{1+\gamma})$, so provided $\gamma<1/2$, one may use Tur\'an's theorem to show that there is
an outcome of $G\left[U\right]$ with $\Omega\left(\sqrt{n}\right)$
different degrees. This fact has some immediate consequences for $|\Phi(G)|$: for example, Bukh and Sudakov observed that one can obtain $\Omega(\sqrt n)$ subgraphs with different numbers of edges simply by choosing different vertices of $U$ to delete from $G[U]$.

There are two straightforward ways one might hope to improve on this simple bound. First, we can repeat the above argument for many different values of $p$, and second, instead of deleting single vertices, we might hope to obtain a richer variety of subgraphs by adding and deleting different combinations of vertices. Narayanan, Sahasrabudhe and Tomon \cite{NST16} combined both these ideas, as follows.

In an $n$-vertex $(c,\gamma)$-diverse graph $G$, first use the pigeonhole principle to identify a set $W_0$ of $\Theta(\sqrt n)$ vertices with degrees contained in a narrow interval $[d,d+\sqrt n]$, for some $d=\Omega(n)$. Then, for $\Theta(\sqrt n)$ well-separated values of $p$, do the following. Let $U$ be a random subset of the vertices not in $W_0$, obtained by including each vertex with probability $p$. Using the diversity of $G$, one may compute that the expected number of pairs of vertices of $W_0$ which have the same degree into $U$ is $O(|W_0|^2/\sqrt n+|W_0|n^\gamma)=O(n^{1/2+\gamma})$, so one can show with Tur\'an's theorem that there is an outcome of $U$ such that $W_0$ contains a subset $W$ of $\Omega(n^{1/2-\gamma})$ vertices with different degrees into $U$. Moreover, since the initial degrees $d(w)$ were chosen to be very similar, one can show that actually the $d_U(w)$ are likely to still lie in an interval of length $O(\sqrt n)$.

Because the degrees of vertices in $W$ are so well-behaved, one can then show that many different values $e(G[U\cup Z])$ can be obtained with different subsets $Z\subseteq W$. Indeed, by varying the number of vertices in $Z$, one can change $e(G[U\cup Z])$ by increments of $\Theta(n)$, and by swapping low-degree vertices with high-degree vertices, one can change $e(G[U\cup Z])$ by increments of about $\sqrt n$. That is to say, by choosing subsets $Z\subseteq W$ of certain types, one can obtain $\Omega(n^{1-2\gamma})$ different values of $e(G[U\cup Z])$ separated by $\Omega(\sqrt n)$ from each other.

The above ideas yield $|\Phi(G)|=\Omega(n^{3/2-2\gamma})$ in a relatively straightforward fashion. Since the diversity lemma of Bukh and Sudakov allows $\gamma$ to be arbitrarily small, this proves that $n$-vertex $O(1)$-Ramsey graphs $G$ have $|\Phi(G)|=n^{3/2-o(1)}$. In order to improve this to $n^{2-o(1)}$, one would ideally like to be able to show that for each of the $\Omega(n^{1-o(1)})$ choices of $Z$ described above, one can add an additional vertex $w\in W$ to $Z$ in $n^{1/2-o(1)}$ different ways to obtain about $\sqrt n$ different values of $e(G[U\cup Z\cup \{w\}])$ that ``fill in'' the interval between consecutive values of $e(G[U\cup Z])$. Unfortunately, while by construction the degrees $d_U(w)$, for $w\in W$, are different, it does not follow that the $d_{U\cup Z}(w)$ are also different. In order to make this approach work, the authors of \cite{NST16} came up with a way to introduce some limited randomness into the choice of the sets $Z$, and with a rather delicate combination of concentration and anticoncentration arguments they were able to show that there are likely to be many different values of $d_{U\cup Z}(w)$.

There are two main obstacles that need to be overcome to prove $|\Phi(G)|=\Omega(n^2)$ with the above strategy. Most obviously, there is a factor of $n^\gamma$ that must be eliminated. Recall that this factor originates from the upper bound $O(|W_0|^2/\sqrt n+|W_0|n^\gamma)$ on the expected number of pairs of vertices of $W_0$ which have the same degree into $U$. It does not seem that this estimate itself can be improved, but the unwanted factor of $n^\gamma$ would disappear if we could arrange for $W_0$ to have size $\Omega(n^{1/2+\gamma})$ instead of size $\Theta(\sqrt n)$. Unfortunately, if we want $W_0$ to be asymptotically larger than $\sqrt n$, it is no longer possible to guarantee that the degrees of vertices in $W_0$ fall within an interval of length $O(\sqrt n)$, and this would cause problems in other parts of the argument. The way we overcome this issue is by allowing two possibilities for the structure of $W_0$. Either $W_0$ is a set of vertices as before, or $W_0$ is a set of disjoint \emph{pairs} of vertices $\{x_1,x_2\}$ with similar values of $d(\{x_1,x_2\})=d(x_1)+d(x_2)$. We may then treat these pairs as we would treat single vertices in the above argument, considering sets $Z$ that are the union of some subset of the pairs in $W_0$. Note that there are $\Theta(n^2)$ pairs of vertices in $G$, but only $O(n)$ possibilities for $d(\{x_1,x_2\})$, so this relaxation gives us a lot of flexibility. This idea actually allows us to take $W_0$ to be of size $\Omega(n^{3/4})$, but also introduces some new complications that must be taken care of. In particular, Bukh and Sudakov's notion of $(c,\gamma)$-diversity is not strong enough to deal with pairs of vertices, so in \cref{sec:tools} we introduce a new notion of \emph{$(\delta,\varepsilon)$-richness}.

The second main obstacle concerns the final part of the argument, where one shows that there are likely to be many different values of $d_{U\cup Z}(w)$ among the $w\in W$. In \cite{NST16}, for this part of the argument the main random set $U$ had already been fixed, so the only source of randomness was the much smaller set $Z$. In this setting, in order to find close to $\sqrt n$ different values of $d_{U\cup Z}(w)$ it does not merely suffice to consider the variation induced by the random set $Z$: one must also take advantage of the separation between different $d_U(w)$, and show that this approximately corresponds to separation between the $d_{U\cup Z}(w)$. It seems that with this approach there is an unavoidable loss of a logarithmic factor, and it is not clear how to prove a result stronger than $|\Phi(G)|=\Omega(n^2/\log n)$.

In the present paper we take a somewhat different approach, with a ``double-exposure'' technique. Specifically, we obtain our random set $U$ as a random subset of about half of the vertices of a larger random set $U_0$. Using the ideas sketched above, we can first use the randomness of $U_0$ to show that there are subsets $Z\subseteq W_0$ which give $\Omega(n)$ different values of $e(G[U_0\cup Z])$ separated by $\Omega(\sqrt n)$ from each other. Then, we can use the randomness of $U\subseteq U_0$ to show that for most $Z$ there are $\Omega(\sqrt n)$ different values of $d_{U\cup Z}(w)$, leading to $\Omega(\sqrt n)$ different values of $e(G[U\cup Z\cup \{w\}])$ closely clustered around $e(G[U\cup Z])$. Of course, we also need to show that this second round of randomness did not cause too much damage to the separation we established with the first round: we need to show that the $e(G[U\cup Z])$ are likely to be well-separated from each other, using the fact that the $e(G[U_0\cup Z])$ were chosen to be well-separated from each other. This can be done by taking advantage of the particular structure of the sets $Z$, and considering an appropriate notion of what it means for a sequence of values to be ``well-separated''.

\section{Basic tools}\label{sec:tools}

In this section we give a number of general results which will be useful in the proof of \cref{conj:NST}. Some of these are well-known, and some are new.

First, as mentioned in the introduction, the following lemma is due to Erd\H os
and Szemer\'edi \cite{ES72}.
\begin{lem}
\label{lem:ES}For any $C$ there exists $\varepsilon>0$ such that
every $C$-Ramsey graph has edge density between $\varepsilon$ and
$1-\varepsilon$.
\end{lem}

Next, we need the notion of \emph{diversity}, introduced by Bukh and
Sudakov \cite{BS07}. Recall from \cref{sec:outline} that an $n$-vertex graph is \emph{$\left(c,\gamma\right)$-diverse}
if for each vertex $x\in V$, we have $\left|N\left(x\right)\triangle N\left(y\right)\right|<cn$
for at most $n^{\gamma}$ vertices $y\in V$. Roughly speaking, this
means the neighbourhoods of most pairs of vertices are very different. As in \cite{BS07}, we will actually only ever need to take $\gamma=1/5$, so we write ``$c$-diverse'' as shorthand for ``$(c,1/5)$-diverse''.

In the same way as \cite{BS07,NST16}, the significance of this notion for us is
that in a diverse graph, if $U$ is a random set of vertices, then
for most pairs of vertices their degrees into $U$ are not too strongly
correlated. In addition to this basic notion of diversity we also
introduce a notion of diversity for \emph{pairs} of vertices. Say
an $n$-vertex graph is \emph{$\left(c,\alpha\right)_{2}$-diverse}
if for each pair of vertices $\boldsymbol{x}=\left\{ x_{1},x_{2}\right\} $
such that $\left|N\left(x_{1}\right)\triangle\overline{N\left(x_{2}\right)}\right|\ge\alpha n$,
one cannot find $n^{1/5}$ other pairs $\boldsymbol{y}=\left\{ y_{1},y_{2}\right\} $,
disjoint to $\boldsymbol{x}$ and each other, such that $\left|N\left(\boldsymbol{x}\right)\triangle N\left(\boldsymbol{y}\right)\right|<cn$
(recall from \cref{subsec:definitions} the non-standard multiset definitions
of $N\left(\boldsymbol{x}\right),N\left(\boldsymbol{y}\right)$ and
$N\left(\boldsymbol{x}\right)\triangle N\left(\boldsymbol{y}\right)$).

In this paper, it will be convenient to deduce diversity from a slightly
stronger condition. Say an $n$-vertex graph is \emph{$\left(\delta,\varepsilon\right)$-rich}
if for any vertex subset $W$ with $\left|W\right|\ge\delta n$, at
most $n^{1/5}$ vertices $v$ have $\left|N\left(v\right)\cap W\right|<\varepsilon\left|W\right|$
or $\left|\overline{N\left(v\right)}\cap W\right|<\varepsilon\left|W\right|$. We remark that a slightly different definition of richness appeared in the published version of this paper, which was not quite suitable for our application. We thank Mantas Baksys and Xuanang Chen for bringing this to our attention.

\begin{lem}
\label{lem:diverse-pairs}Let $G$ be a $\left(\delta,\varepsilon\right)$-rich
graph on a set $V$ of $n$ vertices, with $\delta\le1/2$. Then,
\begin{enumerate}
\item $G$ is $\left(\varepsilon/2\right)$-diverse;
\item $G$ is $\left(\alpha\varepsilon/2,\alpha\right)_{2}$-diverse
for any $\alpha\ge2\delta$;
\item $G$ has at most $n^{1+1/5}$ pairs $\left\{ x_{1},x_{2}\right\} \in\binom{V}{2}$
with $\left|N\left(x_{1}\right)\triangle\overline{N\left(x_{2}\right)}\right|<\left(\varepsilon/2\right)n$.
\end{enumerate}
\end{lem}

\begin{proof}
For the first statement, for each vertex $x$ either $\left|N\left(x\right)\right|\ge n/2$
or $\left|\overline{N\left(x\right)}\right|\ge n/2$. In the former
case, for all but at most $n^{1/5}$ vertices $y$ we have $\left|N\left(x\right)\cap\overline{N\left(y\right)}\right|\ge\varepsilon\left|N\left(x\right)\right|\ge\varepsilon n/2$,
and in the latter case for all but at most $n^{1/5}$ vertices
$y$ we have $\left|\overline{N\left(x\right)}\cap N\left(y\right)\right|\ge\varepsilon\left|\overline{N\left(x\right)}\right|\ge\varepsilon n/2$.
In either case, there are at most $n^{1/5}$ vertices $y$ with
$\left|N\left(x\right)\triangle N\left(y\right)\right|<\varepsilon n/2$,
as desired.

For the second statement, note that if $\left|N\left(x_{1}\right)\triangle\overline{N\left(x_{2}\right)}\right|\ge\alpha n$
then $\left|N\left(x_{1}\right)\cap N\left(x_{2}\right)\right|\ge\left(\alpha/2\right)n$
or $\left|\overline{N\left(x_{1}\right)}\cap\overline{N\left(x_{2}\right)}\right|\ge\left(\alpha/2\right)n$.
Suppose that there were a pair $\boldsymbol{x}$ and a collection
$Y$ of $n^{1/5}$ pairs contradicting $\left(\alpha\varepsilon/2,\alpha\right)_{2}$-diversity,
and suppose without loss of generality that $\left|N\left(x_{1}\right)\cap N\left(x_{2}\right)\right|\ge\left(\alpha/2\right)n\ge\delta n$.
Then, for each vertex $y$ in each $\boldsymbol{y}\in Y$,
\[
\left|\overline{N\left(y\right)}\cap N\left(x_{1}\right)\cap N\left(x_{2}\right)\right|\le\left|N\left(\boldsymbol{x}\right)\triangle N\left(\boldsymbol{y}\right)\right|<\left(\alpha\varepsilon/2\right)n\le\varepsilon\left|N\left(x_{1}\right)\cap N\left(x_{2}\right)\right|,
\]
and the set of all such $y$ would contradict $\left(\delta,\varepsilon\right)$-richness.

For the third statement, we will show that for each of the $n$ choices
of $x_{1}$ there are at most $n^{1/5}$ pairs $\left\{ x_{1},x_{2}\right\} \in\binom{V}{2}$
with $\left|N\left(x_{1}\right)\triangle\overline{N\left(x_{2}\right)}\right|<\left(\varepsilon/2\right)n$.
Consider any vertex $x_{1}$ and suppose without loss of generality
that $\left|N\left(x_{1}\right)\right|\ge n/2$. There are at most
$n^{1/5}$ vertices $x_{2}$ with $N\left(x_{2}\right)\cap N\left(x_{1}\right)<\varepsilon n/2$,
and for all other $x_{2}$ we have $\left|N\left(x_{1}\right)\triangle\overline{N\left(x_{2}\right)}\right|\ge\left|N\left(x_{2}\right)\cap N\left(x_{1}\right)\right|\ge\left(\varepsilon/2\right)n$.
\end{proof}
Now, we show that every $C$-Ramsey graph contains a rich induced
subgraph of linear size. The proof approach is based on a related
lemma due to Bukh and Sudakov \cite[Lemma~2.2]{BS07}, which in turn
uses ideas from \cite{She98,PR99}.
\begin{lem}
\label{lem:rich}For any $C,\delta>0$, there exist $\varepsilon=\varepsilon\left(C\right)>0$
and $c=c\left(C,\delta\right)>0$ and $n_0=n_0(\delta)$ such that if $n\ge n_0$ then every $n$-vertex $C$-Ramsey
graph contains a $\left(\delta,\varepsilon\right)$-rich induced subgraph
on at least $cn$ vertices.
\end{lem}

\begin{proof}
Suppose for the purpose of contradiction that every set of at least
$cn$ vertices fails to induce a $\left(\delta,\varepsilon\right)$-rich
subgraph, for $c,\varepsilon$ to be determined. For some large $K=K\left(C\right)$
to be determined, we will inductively construct a sequence of induced
subgraphs $G=G\left[U_{0}\right]\supseteq G\left[U_{1}\right]\supseteq\dots\supseteq G\left[U_{K}\right]$
and disjoint vertex sets $S_{1},\dots,S_{K}$ such that for all $i$,
$\left|U_{i}\right|\ge\left(\delta/4\right)\left|U_{i-1}\right|$,
$\left|S_{i}\right|=\left(cn\right)^{1/5}/2$, $S_{i}\subseteq U_{i-1}$,
and
\begin{equation}
\left[d\left(S_{i},S_{j}\right)<4\varepsilon\text{ for all }j>i\right]\text{ or }\left[d\left(S_{i},S_{j}\right)>1-4\varepsilon\text{ for all }j>i\right].\label{eq:density-cases}
\end{equation}
This will suffice, as follows. Without loss of generality suppose
that the first case of \cref{eq:density-cases} holds for at least
half of the choices of $i$, and let $S$ be the union of the corresponding
$S_{i}$. Then one can compute $d\left(S\right)<4\varepsilon+2/K$.
For sufficiently small $\varepsilon$, large $K$ and large $n_0$, this density
is too low for $G\left[S\right]$ not to contain a homogeneous subgraph
of size $C\log n$, by \cref{lem:ES}.

Let $U_{0}=V\left(G\right)$. For $1\le i\le K$ we will construct
$U_{i},S_{i}$, assuming $U_{0},\dots,U_{i-1},\,S_{1},\dots,S_{i-1}$
have already been constructed. For $c\le (\delta/4)^K$ we have $\left|U_{i-1}\right|\ge cn$, so by assumption $U_{i-1}$ contains a set $W$
of at least $\delta\left|U_{i-1}\right|$ vertices and a set $Y$
of $\left(cn\right)^{1/5}$ vertices contradicting $\left(\delta,\varepsilon\right)$-richness.
Suppose without loss of generality that $\left|N_{U_{i-1}}\left(v\right)\cap W\right|\le\varepsilon\left|W\right|$
for half the vertices $v\in Y$, and let $S_{i}$ be the corresponding
subset of $Y$. Then, let $U=W\backslash S_{i}$, so $\left|U\right|\ge\left|W\right|/2$,
and let $U_{i}\subseteq U$ be the set of vertices $v\in U$ with
$d\left(\left\{ v\right\} ,S_{i}\right)\le4\varepsilon$. Now, we
just need to show $\left|U_{i}\right|\ge\left(\delta/4\right)\left|U_{i-1}\right|$.
To this end, first observe that for all $y\in S_{i}$ we have $d\left(\left\{ y\right\} ,U\right)=d_{U}\left(y\right)/\left|U\right|\le\left(\varepsilon\left|W\right|\right)/\left(\left|W\right|/2\right)=2\varepsilon$.
Then,
\[
4\varepsilon\left|U\backslash U_{i}\right|<\sum_{v\in U\backslash U_{i}}d\left(\left\{ v\right\} ,S_{i}\right)\le\frac{e\left(U,S_{i}\right)}{\left|S_{i}\right|}=\frac{\left|U\right|}{\left|S_{i}\right|}\sum_{y\in S_{i}}d\left(\left\{ y\right\} ,U\right)\le2\varepsilon\left|U\right|,
\]
implying that $\left|U_{i}\right|>\left|U\right|/2\ge\left(\delta/4\right)\left|U_{i-1}\right|$,
as desired.
\end{proof}
Next, we will use a very slight variation of the Erd\H os\textendash Littlewood\textendash Offord
theorem. Say a random variable is of \emph{$\left(n,p\right)$-Littlewood\textendash Offord
type} if it can be expressed in the form $X=a_{1}\xi_{1}+\dots+a_{n}\xi_{n}+C$,
where $a_{1},\dots,a_{n}\in\ZZ\backslash\left\{ 0\right\} $ and $C\in\ZZ$
are fixed and $\xi_{0},\dots,\xi_{n}$ are independent, identically
distributed $p$-Bernoulli random variables (taking the value 1 with
probability $p$ and the value 0 with probability $1-p$). The following
variation of the Erd\H os\textendash Littlewood\textendash Offord
theorem follows from, for example, \cite[Lemma~A.1]{BVW10}.
\begin{lem}
\label{lem:ELO}Suppose $X$ is of $\left(n,p\right)$-Littlewood\textendash Offord
type, for $p=\Omega\left(1\right)$ and $1-p=\Omega\left(1\right)$.
Then for any $x\in\ZZ$, $\Pr\left(X=x\right)=O\left(1/\sqrt{n}\right).$
\end{lem}

Finally, throughout the proof we will use Markov's inequality, Chebyshev's
inequality, the Chernoff bound, the Azuma\textendash Hoeffding inequality
and Tur\'an's theorem. Statements and proofs of all of these can
be found, for example, in \cite{AS}.

\section{Proof of \texorpdfstring{\cref{conj:NST}}{Theorem~\ref{conj:NST}}}

Very broadly, as outlined in \cref{sec:outline}, the basic idea of our proof is similar to the proof
in \cite{NST16}. We will find many induced subgraphs $G[U]$ with ``well-separated''
numbers of edges, and we will augment these subgraphs in many different
ways. To be precise, \cref{conj:NST} will be an immediate consequence
of the following lemma.
\begin{lem}
\label{lem:per-m}For any $C$ there is $c=c(C)>0$ such that the following
holds. For any $n$-vertex $C$-Ramsey graph $G$ and any $m$ satisfying
$cn^{2}\le m\le2cn^{2}$, there are disjoint subsets $U,W\subseteq V$
with $\left|e\left(U\right)-m\right|=O\left(n^{3/2}\right)$ and $\left|W\right|=O\left(\sqrt{n}\right)$
such that
\[
\left|\left\{ e\left(U\cup Z\right):Z\subseteq W\right\} \right|=\Omega\left(n^{3/2}\right).
\]
\end{lem}

\begin{proof}[Proof of \cref{conj:NST} given \cref{lem:per-m}]
For any $U,W$ as in \cref{lem:per-m}, we have 
\[
\left|e\left(U\cup W\right)-e\left(U\right)\right|=O\left(n^{3/2}\right),
\]
because $\left|W\right|=O\left(\sqrt{n}\right)$ and each $w\in W$
has fewer than $n$ neighbours in $U\cup W$. That is to say, the
$\Omega\left(n^{3/2}\right)$ values of $e\left(U\cup Z\right)$ are
contained in an interval of length $O\left(n^{3/2}\right)$ centered
at $e\left(U\right)$. By applying \cref{lem:per-m} to $\Omega\left(\sqrt{n}\right)$
values of $m$ each separated by a sufficiently large multiple of
$n^{3/2}$, we therefore get $\Omega\left(n^{2}\right)$ different
subgraph sizes.
\end{proof}
The sets $W$ in \cref{lem:per-m} will be comprised of multiple disjoint
subsets $S,T,X$ with different roles. Roughly speaking, given some
$Z\subseteq W$ containing exactly one element of $X$, we will be
able to increase the number of edges in $e\left(U\cup Z\right)$ by
$\Theta\left(n\right)$ by adding an element of $S$ to $Z$, we will
be able to increase the number of edges by $\Theta\left(\sqrt{n}\right)$
by exchanging an element of $S$ in $Z$ with an element of $T$,
and we will be able to modify $e\left(U\cup Z\right)$ very finely
by $\Theta\left(\sqrt{n}\right)$ different amounts by making different
choices for the single element of $X$ in $Z$. With different combinations
of these operations we will be able to obtain $\Omega\left(n^{3/2}\right)$
different values of $e\left(U\cup Z\right)$.

As outlined in \cref{sec:outline}, due to some technical obstacles we were not actually able
to construct sets $U,S,T,X$ that give us control over subgraph sizes
in such a simplistic way. Perhaps our most important new idea, which
gives us a lot of flexibility, is that we may allow $S,T,X$ to be
sets of disjoint \emph{pairs} of vertices rather than just vertex
sets. The following lemma will be a starting point for our construction.
\begin{lem}
\label{lem:sub-per-m}For any $C$ there is $c>0$ such that the following holds. For
any $n$-vertex $C$-Ramsey graph $G$ and any $m$ satisfying $cn^{2}\le m\le2cn^{2}$,
there is a vertex set $U_{0}\subseteq V$ with $\left|e\left(U_{0}\right)-4m\right|=O\left(n^{3/2}\right)$
and sets $S,T,X$ of size $\Theta\left(\sqrt{n}\right)$ such that
\begin{enumerate}
\item either $U_{0},S,T,X\subseteq V$ are disjoint sets of vertices or
$S,T,X\subseteq\binom{V}{2}$ are sets of disjoint \emph{pairs} of
vertices such that no vertex appears in more than one of $U_{0},S,T,X$;
\item there is $d=\Theta\left(n\right)$ such that $d_{U_{0}}\left(\boldsymbol{x}\right)=d+O\left(\sqrt{n}\right)$
for each $\boldsymbol{x}\in S\cup T\cup X$;
\item the degrees from $S$ into $U_{0}$ are smaller by $\Omega\left(\sqrt n\right)$
than the degrees from $T$ into $U_{0}$ (that is, $\min_{\boldsymbol{x}\in T}d_{U_{0}}\left(\boldsymbol{x}\right)-\max_{\boldsymbol{x}\in S}d_{U_{0}}\left(\boldsymbol{x}\right)=\Omega\left(\sqrt n\right)$);
\item for each $\left\{ \boldsymbol{x},\boldsymbol{y}\right\} \in\binom{X}{2}$,
we have $\left|N_{U_{0}}\left(\boldsymbol{x}\right)\triangle N_{U_{0}}\left(\boldsymbol{y}\right)\right|=\Omega\left(n\right)$.
\end{enumerate}
\end{lem}

Note that when we write a variable name in bold, it may be a single
vertex or a pair of vertices. Also, we emphasise that we are thinking of $C$ as a fixed constant, so the constants implied by the asymptotic notation in \cref{lem:sub-per-m} may depend on $C$ (but nothing else).

We will prove \cref{lem:sub-per-m} in \cref{sec:sub-per-m}. Without
going into too much detail about the proof, the idea is to first apply \cref{lem:rich} to reduce to an induced subgraph with rich neighbourhoods, then use the pigeonhole principle to find a large set $L$ of either vertices or pairs of vertices with very similar degrees. Then, we choose $U_{0}$ randomly and choose $S,T,X\subseteq L$ based on this random
outcome to satisfy the properties in \cref{lem:sub-per-m}.

Now, consider a $C$-Ramsey graph $G$, let $c$ be as in \cref{lem:sub-per-m}, and consider some $m$ satisfying $cn^2\le m\le 2cn^2$. Apply \cref{lem:sub-per-m} to obtain sets $U_{0},S,T,X$, and let $c'$ be a constant such that
\begin{equation}\min_{\boldsymbol{x}\in T}d_{U_{0}}\left(\boldsymbol{x}\right)-\max_{\boldsymbol{x}\in S}d_{U_{0}}\left(\boldsymbol{x}\right)\ge8c'\sqrt n.\label{eq:sep}\end{equation}
(Such a constant exists by the thid property of \cref{lem:sub-per-m}).

Fix an ordering of the elements of $S$ and of $T$, and let $\mathcal{P}$
be the set of pairs $\left(k,i\right)\in\ZZ^{2}$ with $c'\sqrt{n}\le k\le2c'\sqrt{n}$ and $0\le i\le c'\sqrt n$. For each $\left(k,i\right)\in\mathcal{P}$, define the set $Z_{k,i}$
to contain the vertices of the first $k-i$ elements from $S$ and
the first $i$ elements from $T$. Note that $e\left(Z_{k,0}\cup U_{0}\right)-e\left(Z_{k-1,0}\cup U_{0}\right)=\Theta\left(n\right)$
for each $k$, because $d_{U_{0}}\left(\boldsymbol{x}\right)=\Theta\left(n\right)$
for each $\boldsymbol{x}\in S$ by the second property of \cref{lem:sub-per-m}.
Also note that $e\left(Z_{k,i}\cup U_{0}\right)-e\left(Z_{k,i-1}\cup U_{0}\right)=\Theta\left(\sqrt{n}\right)$
for each $k,i$, by the second property of \cref{lem:sub-per-m}, \cref{eq:sep} and the fact that $e\left(Z_{k,i}\right)-e\left(Z_{k,i-1}\right)\le 2c'\sqrt{n}$.
Therefore, as $\left(k,i\right)$ varies lexicographically, the $e\left(Z_{k,i}\cup U_{0}\right)$
comprise $\Omega\left(n\right)$ roughly evenly spaced values.

Now, let $U$ be a random subset of $U_0$, where each vertex is present with probability $1/2$ independently. We would like some approximation of the spacing described above still to hold for the
collection of random values $e\left(U_{k,i}\right)$, where $U_{k,i}=Z_{k,i}\cup U$.
Also, we want to use the randomness of $U$ to show that for most
$k,i$, the $\boldsymbol{x}\in X$ have $\Omega\left(\sqrt{n}\right)$
different degrees $d_{U_{k,i}}\left(\boldsymbol{x}\right)$ into $U_{k,i}$.
In \cref{sec:per-k} we will prove the following lemma, from which
\cref{lem:per-m} will easily follow.
\begin{lem}
\label{lem:per-k}Let $G$ be a $C$-Ramsey graph, let $S,T,X,d$ be obtained from an application of \cref{lem:sub-per-m}, and let $Z_{k,i},U,U_{k,i}$ be as defined above.
Then, there are constants $M,\beta,Q>0$ (with $Q\ge 3\beta$) such that for
each $c'\sqrt{n}\le k\le2c'\sqrt{n}$, the following hold.
\begin{enumerate}
\item With probability at least 0.99, there is a set $\mathcal{I}_{k}$
of $\left(1-\beta/\left(2M\right)\right)c'\sqrt{n}$ values of $i$
with the following property. For each $i\in\mathcal{I}_{k}$, there
is a set $X_{k,i}\subseteq X$ of size $\Omega\left(\sqrt{n}\right)$
such that the $d_{U_{k,i}}\left(\boldsymbol{x}\right)$, for $\boldsymbol{x}\in X_{k,i}$,
are distinct, and all lie in the interval between $d/2-Q\sqrt{n}$
and $d/2+Q\sqrt{n}$.
\item Let $e_{k,i}=e\left(Z_{k,i}\right)+e\left(Z_{k,i},U\right)=e\left(U_{k,i}\right)-e\left(U\right)$.
With probability at least 0.99,
\[
\left|e_{k,0}-\E e_{k,0}\right|\le Qn.
\]
\item With probability at least $1/2$, $e_{k,c'\sqrt{n}}-e_{k,0}\ge3\beta n$.
\item Let $\Delta_{k,i}=e_{k,i}-e_{k,i-1}$. With probability at least 0.99,
\[
\sum_{\left|\Delta_{k,i}\right|\ge M\sqrt{n}}\left|\Delta_{k,i}\right|\le\beta n.
\]
(That is to say, the ``unusually large'' increments $\Delta_{k,i}$
have low total volume).
\end{enumerate}
\end{lem}

Again we emphasise that we are thinking of $C$ as fixed, so $M,\beta,Q$ may depend on $C$, via $c'$ and the constants implied by the asymptotic notation in \cref{lem:sub-per-m}.

Now we can prove \cref{lem:per-m} given \cref{lem:sub-per-m,lem:per-k}.
\begin{proof}[Proof of \cref{lem:per-m}]
Apply \cref{lem:sub-per-m} to obtain $U_{0},S,T,X$, and define $\mathcal{P},U,Z_{k,i},U_{k,i}$
as above. We will prove the statement of the lemma for $W$ being
the set of all the vertices in $S\cup T\cup X$.

For each $c'\sqrt{n}\le k\le2c'\sqrt{n}$, with probability at least
0.4, all four parts of \cref{lem:per-k} are satisfied. Let $\mathcal{K}$
be the set of such $k$. Then $\E\left[c'\sqrt{n}-\left|\mathcal{K}\right|\right]\le\left(0.6\right)c'\sqrt{n}$,
so by Markov's inequality, with probability at least $1-0.6/0.8=1/4$
we have $\left|\mathcal{K}\right|\ge\left(0.2\right)c'\sqrt{n}$.
Also, note that $\left|\E e\left(U\right)-m\right|=O\left(n^{3/2}\right)$
and there are $O\left(n^{3}\right)$ pairs of edges in $G$ whose
presence in $G\left[U\right]$ are dependent (this can only occur
if they share a vertex). So, $\Var e\left(U\right)=O\left(n^{3}\right)$
and by Chebyshev's inequality, $\left|m-e\left(U\right)\right|=O\left(n^{3/2}\right)$
with probability at least 0.9. Fix an outcome of $U$ satisfying both
these events, which hold together with probability at least $0.15$.

For each $k$, we have 
\[
\E e_{k,0}-\E e_{k-1,0}=\left(e\left(Z_{k,0}\cup U_{0}\right)-e\left(Z_{k-1,0}\cup U_{0}\right)\right)/2=\Theta\left(n\right),
\]
since $\E d_{U}\left(\boldsymbol{x}\right)=d_{U_{0}}\left(\boldsymbol{x}\right)/2$
for each $\boldsymbol{x}\in S$. Let $\mathcal{K}'\subseteq\mathcal{K}$
contain every $q$th element of $\mathcal{K}$, for sufficiently large
$q$ such that the values of $\E e_{k,0}$, for $k\in\mathcal{K}'$,
are separated by at least $4Qn$. Then, $\left|\mathcal{K}'\right|=\Theta\left(\sqrt{n}\right)$,
and by part 2 of \cref{lem:per-k}, the values of $e_{k,0}$, for $k\in\mathcal{K}'$,
are separated by at least $2Qn$.

Next, by part 3 of \cref{lem:per-k} we have $e_{k,c'\sqrt{n}}-e_{k,0}\ge3\beta n$
for each $k\in\mathcal{K}'$. Consider the range of integers between
$e_{k,0}$ and $e_{k,0}+3\beta n$, and within this range consider
$2\beta\sqrt{n}/M$ intervals of length $M\sqrt{n}$, each separated
by a distance of $M\sqrt{n}/2=\Omega\left(\sqrt{n}\right)$. By part
4 of \cref{lem:per-k}, at most $\beta\sqrt{n}/M$ of these intervals
contain no value of $e_{k,i}$. Consider a representative $e_{k,i}$
from $\beta\sqrt{n}/M$ different intervals, and let $\mathcal{I}_{k}'$
be the set of corresponding indices $i$. Let $\mathcal{I}_{k}''=\mathcal{I}_{k}\cap\mathcal{I}_{k}'$,
so that $\left|\mathcal{I}_{k}''\right|\ge\beta\sqrt{n}/M-\beta c'\sqrt{n}/\left(2M\right)\ge\beta\sqrt{n}/\left(2M\right)$.
By construction, $\left|e_{k,i}-e_{k,0}\right|\le3\beta n$ for each
$i\in\mathcal{I}_{k}''$, and the values of $e_{k,i}$, for $i\in\mathcal{I}_{k}''$,
are separated by $\Omega\left(\sqrt{n}\right)$. We are assuming that
$Q\ge3\beta$, so among choices of $k\in\mathcal{K}'$, $i\in\mathcal{I}_{k}''$,
we already have a total of $\Omega\left(n\right)$ different values
of $e\left(U_{k,i}\right)$ separated by $\Omega\left(\sqrt{n}\right)$.

Now, consider every $q$th of these values (in increasing order),
for sufficiently large $q$ such that each resulting pair of values
are separated by at least $2Q\sqrt{n}$. Let $\mathcal{P}'\subseteq\left\{ \left(k,i\right):k\in\mathcal{K}',i\in\mathcal{I}_{k}''\right\} $
be the corresponding set of indices (so $\left|\mathcal{P}'\right|=\Omega\left(n\right)$).
For each $\left(k,i\right)\in\mathcal{P}'$, we have $i\in\mathcal{I}_{k}$,
so by part 1 of \cref{lem:per-k}, there is $X_{k,i}$
such that the values $d_{U_{k,i}}\left(\boldsymbol{x}\right)$, for
$\boldsymbol{x}\in X_{k,i}$, are all different, yet are all in a fixed
interval of length $2Q\sqrt{n}$. Therefore, among choices of $\left(k,i\right)\in\mathcal{P}'$
and $\boldsymbol{x}\in X_{k,i}$, there are $\Omega\left(n^{3/2}\right)$
different values of $e\left(U_{k,i}\cup\boldsymbol{x}\right)=e\left(U\cup\left(Z_{k,i}\cup\boldsymbol{x}\right)\right)$,
as desired.
\end{proof}

\subsection{\label{sec:sub-per-m}Proof of \texorpdfstring{\cref{lem:sub-per-m}}{Lemma~\ref{lem:sub-per-m}}}
\begin{proof}[Proof of \cref{lem:sub-per-m}]
First, consider $\varepsilon=\varepsilon\left(C\right)$ from \cref{lem:rich}
and note that we can assume $G$ is $\left(\delta,\varepsilon\right)$-rich,
for $\delta=\varepsilon/4$. To see this, first apply \cref{lem:rich}
to obtain a $\Omega\left(n\right)$-vertex $\left(\delta,\varepsilon\right)$-rich
induced subgraph $G\left[V'\right]\subseteq V$ . Since $\log\left|V'\right|\ge\left(1/2\right)\log n$,
$G\left[V'\right]$ is still $2C$-Ramsey, so by tweaking some constants
it suffices to find our desired sets $U_{0},S,T,X$ inside $G\left[V'\right]$.

So, we make the aforementioned richness assumption. By \cref{lem:diverse-pairs}
with $\alpha=\varepsilon/2$, this means that $G$ is both $\varepsilon/2$-diverse
and $\left(\varepsilon^{2}/4,\varepsilon/2\right)_{2}$-diverse,
and there are at most $n^{1+1/5}$ pairs of vertices $\left\{ x_{1},x_{2}\right\} $
with $\left|N\left(x_{1}\right)\triangle\overline{N\left(x_{2}\right)}\right|<\left(\varepsilon/2\right)n$.
Note that each of the $\Omega\left(n^{2}\right)$ sums $d\left(\boldsymbol{x}\right)=d\left(x_{1}\right)+d\left(x_{2}\right)$,
for $\boldsymbol{x}=\left\{ x_{1},x_{2}\right\} \in\binom{V}{2}$,
lie between 0 and $2n$, so by the pigeonhole principle there is some
$d'$ and a collection of $\Omega\left(n^{3/2}\right)$ pairs $H\subseteq\binom{V}{2}$
such that $d\left(\boldsymbol{x}\right)=d'+O\left(\sqrt{n}\right)$
for all $\boldsymbol{x}\in H$. Interpret $H$ as a graph on the vertex
set $V$ with $\Omega\left(n^{3/2}\right)$ edges, and obtain a further
graph $H'$ by deleting the $O\left(n^{1+1/5}\right)=o\left(n^{3/2}\right)$
edges $\left\{ x_{1},x_{2}\right\} $ with $\left|N\left(x_{1}\right)\triangle\overline{N\left(x_{2}\right)}\right|<\left(\varepsilon/2\right)n$.
Now, $H'$ either has a vertex $v$ with $d\left(v\right)=\Omega\left(n^{3/4}\right)$
or it has a matching with $\Omega\left(n^{3/4}\right)$ edges. In
the former case let $d''=d'-d\left(v\right)$ and let $L\subseteq N_{H}\left(v\right)$
be a set of $\Omega\left(n^{3/4}\right)$ neighbours of $v$ in $H$.
In the latter case let $d''=d'$ and let $L$ be a set of $\Omega\left(n^{3/4}\right)$
pairs comprising a matching in $H'$. In both cases, for each $\boldsymbol{x}\in L$,
we have $d\left(\boldsymbol{x}\right)=d''+O\left(\sqrt{n}\right)$.

Next, let $F\subseteq\binom{L}{2}$ be the set of $\left\{ \boldsymbol{x},\boldsymbol{y}\right\} \in\binom{L}{2}$
with $\left|N\left(\boldsymbol{x}\right)\triangle N\left(\boldsymbol{y}\right)\right|<\left(\varepsilon^{2}/4\right)n$.
By one of our two diversity assumptions, interpreting $F$ as a graph,
it has $\left|L\right|=\Omega\left(n^{3/4}\right)$ vertices and maximum
degree at most $n^{1/5}$, so by Tur\'an's theorem
it has an independent set $A$ with size $\Omega\left(\left|L\right|/n^{1/5}\right)=\Omega\left(\sqrt{n}\right)$.
That is to say, for every $\left\{ \boldsymbol{x},\boldsymbol{y}\right\} \in\binom{A}{2}$,
we have $\left|N\left(\boldsymbol{x}\right)\triangle N\left(\boldsymbol{y}\right)\right|=\Omega\left(n\right)$.

Now, by \cref{lem:ES} and the $C$-Ramsey property, $e\left(G\right)\ge800cn^{2}$
for some $c>0$. For $cn^{2}\le m\le2cn^{2}$, let $p=\sqrt{4m/e\left(G\right)}$
(so $p=\Omega\left(1\right)$ and $p\le0.1$), and let $U_{0}$ be
a random subset of $V$ obtained by including each element with probability
$p$ independently. We make a few observations.
\begin{claim*}
The following five events each hold with probability greater than
$4/5$.
\begin{enumerate}
\item $\left|e\left(U_{0}\right)-4m\right|=O\left(n^{3/2}\right)$;
\item there is $Q\subseteq A$ involving no vertices of $U_{0}$, with $\left|Q\right|\ge\left(2/3\right)\left|A\right|$;
\item there is $R\subseteq A$ with $\left|R\right|\ge\left(2/3\right)\left|A\right|$
and $d_{U_{0}}\left(\boldsymbol{x}\right)=pd''+O\left(\sqrt{n}\right)$
for each $\boldsymbol{x}\in R$;
\item $\left|N_{U_{0}}\left(\boldsymbol{x}\right)\triangle N_{U_{0}}\left(\boldsymbol{y}\right)\right|=\Omega\left(n\right)$
for each $\left\{ \boldsymbol{x},\boldsymbol{y}\right\} \in\binom{A}{2}$;
\item the equality $d_{U_{0}}\left(\boldsymbol{x}\right)=d_{U_{0}}\left(\boldsymbol{y}\right)$
holds for $O\left(\sqrt{n}\right)$ pairs $\left\{ \boldsymbol{x},\boldsymbol{y}\right\} \in\binom{A}{2}$.
\end{enumerate}
\end{claim*}
\begin{proof}[Proof of claim]
For the first property, note that $\E e\left(G\left[U_{0}\right]\right)=4m$
and and there are $O\left(n^{3}\right)$ pairs of edges in $G$ whose
presence in $G\left[U_{0}\right]$ are dependent (this can only occur
if they share a vertex), so $\Var e\left(G\left[U_{0}\right]\right)=O\left(n^{3}\right)$.
The desired result then follows from Chebyshev's inequality, for a
sufficiently large constant implicit in ``$O\left(n^{3/2}\right)$''.

For the second property, note that the size of the subset $Q\subseteq A$
of elements of $A$ which contain no vertices of $U_{0}$ has mean
at least $\left(1-p\right)^{2}\left|A\right|$ and variance $O\left(\left|A\right|\right)$;
since $1-p\ge0.9$ the desired result again follows from Chebyshev's
inequality.

For the third property, for each $\boldsymbol{x}\in A$ we have $\E d_{U_{0}}\left(\boldsymbol{x}\right)=pd''+O\left(\sqrt{n}\right)$
and $\Var d_{U_{0}}\left(\boldsymbol{x}\right)=O\left(n\right)$,
so with at probability at least 0.99 we have $d_{U_{0}}\left(\boldsymbol{x}\right)=pd''+O\left(\sqrt{n}\right)$.
Let $R$ be the set of $\boldsymbol{x}$ satisfying this bound; we
have $\E\left|A\backslash R\right|=\left(0.01\right)\left|A\right|$,
so by Markov's inequality, with probability at least $4/5$ we have
$\left|A\backslash R\right|\le\left(1/3\right)\left|A\right|$.

For the fourth property, recall that $\left|N\left(\boldsymbol{x}\right)\triangle N\left(\boldsymbol{y}\right)\right|=\Omega\left(n\right)$
for each $\left\{ \boldsymbol{x},\boldsymbol{y}\right\} \in\binom{A}{2}$.
Note that $N_{U_{0}}\left(\boldsymbol{x}\right)\triangle N_{U_{0}}\left(\boldsymbol{y}\right)=\left(N\left(\boldsymbol{x}\right)\triangle N\left(\boldsymbol{y}\right)\right)\cap U_{0}$,
so that $\left|N_{U_{0}}\left(\boldsymbol{x}\right)\triangle N_{U_{0}}\left(\boldsymbol{y}\right)\right|$
has a binomial distribution with parameters $\left|N\left(\boldsymbol{x}\right)\triangle N\left(\boldsymbol{y}\right)\right|$
and $p$. Then, $\Pr\left(\left|N_{U_{0}}\left(\boldsymbol{x}\right)\triangle N_{U_{0}}\left(\boldsymbol{y}\right)\right|<\left(p/2\right)\left|N\left(\boldsymbol{x}\right)\triangle N\left(\boldsymbol{y}\right)\right|\right)=e^{-\Omega\left(n\right)}$
by the Chernoff bound, and the desired result follows from the union
bound.

For the fifth property, for $\left\{ \boldsymbol{x},\boldsymbol{y}\right\} \in\binom{A}{2}$,
note that the random variable $d_{U_{0}}\left(\boldsymbol{x}\right)-d_{U_{0}}\left(\boldsymbol{y}\right)$
is of $\left(\left|N\left(\boldsymbol{x}\right)\triangle N\left(\boldsymbol{y}\right)\right|,p\right)$-Littlewood\textendash Offord
type. So, recalling that $\left|N\left(\boldsymbol{x}\right)\triangle N\left(\boldsymbol{y}\right)\right|=\Omega\left(n\right)$,
we have $\Pr\left(d_{U_{0}}\left(\boldsymbol{x}\right)=d_{U_{0}}\left(\boldsymbol{y}\right)\right)=O\left(1/\sqrt{n}\right)$.
The expected number of pairs $\left\{ \boldsymbol{x},\boldsymbol{y}\right\} \in\binom{A}{2}$
satisfying $d_{U_{0}}\left(\boldsymbol{x}\right)=d_{U_{0}}\left(\boldsymbol{y}\right)$
is therefore $O\left(\sqrt{n}\right)$, and the desired result follows
from Markov's inequality.
\end{proof}
Fix an outcome of $U_{0}$ satisfying all 5 of the above properties,
and arbitrarily divide $R\cap Q$, which has size at least $\left|A\right|/3$,
into two subsets $Y$ and $X$ of size $\Omega\left(\sqrt{n}\right)$.
Consider the graph on the vertex set $Y$ of all $\left\{ \boldsymbol{x},\boldsymbol{y}\right\} \in\binom{Y}{2}$
with $d_{U_{0}}\left(\boldsymbol{x}\right)=d_{U_{0}}\left(\boldsymbol{y}\right)$.
This graph has $O\left(\sqrt{n}\right)$ edges, so by Tur\'an's theorem
it has an independent set $B$ of size $\Omega\left(\sqrt{n}\right)$.
Order the $\boldsymbol{x}\in B$ by $d_{U_{0}}\left(\boldsymbol{x}\right)$,
let $S$ be the first $\left|B\right|/3$ elements in this ordering
and let $T$ be the last $\left|B\right|/3$ elements. Since each
such $d_{U_{0}}\left(\boldsymbol{x}\right)$ is distinct, this means
$\min_{\boldsymbol{x}\in T}d_{U_{0}}\left(\boldsymbol{x}\right)-\max_{\boldsymbol{x}\in S}d_{U_{0}}\left(\boldsymbol{x}\right)\ge |B|/3=\Omega\left(\sqrt n\right)$.
Let $d=pd''$ and note that $d=\Omega\left(n\right)$, because otherwise
it would be impossible to simultaneously satisfy properties 3 and
4.
\end{proof}

\subsection{\label{sec:per-k}Proof of \texorpdfstring{\cref{lem:per-k}}{Lemma~\ref{lem:per-k}}}
\begin{proof}[Proof of \cref{lem:per-k}]
The constants $\beta,M,Q$ will be determined in that order, in terms of each other. Therefore it is convenient to prove the four parts of \cref{lem:per-k} in a slightly different order than they are stated.

For the third part, note that $\left|e\left(Z_{k,i}\right)-e\left(Z_{k,i-1}\right)\right|\le2c'\sqrt{n}$, and recall \cref{eq:sep}. We have
\[
\E\left[e_{k,i}-e_{k,i-1}\right]\ge\frac{1}{2}\left(e\left(Z_{k,i},U_{0}\right)-e\left(Z_{k,i-1},U_{0}\right)\right)-2c'\sqrt{n}\ge 2c'\sqrt n,
\]
for each $i$. Let $\beta=2(c')^2/3$, so $\Delta_{k}:=e_{k,c'\sqrt{n}}-e_{k,0}$
has expectation at least $3\beta n$. But $\Delta_{k}$
is of $\left(\Omega\left(n\right),1/2\right)$-Littlewood\textendash Offord
type and is therefore symmetrically distributed around its expectation.
The desired result follows.

For the fourth part, note that $\Delta_{k,i}$ has mean $O\left(\sqrt{n}\right)$
and is affected by 1 or 2 by the addition or removal of an element
to/from $U$. So, by the Azuma\textendash Hoeffding inequality, $\Pr\left(\left|\Delta_{k,i}\right|\ge t\right)=\exp\left(-\Omega\left(t^{2}/n\right)\right)$.
Now, for any nonnegative integer random variable $\xi$, we have $\E\xi=\sum_{t=1}^{\infty}\Pr\left(\xi\ge t\right)$,
so
\begin{align*}
\E\left[\left|\Delta_{k,i}\right|\one_{\left|\Delta_{k,i}\right|\ge M\sqrt{n}}\right] & =\sum_{t=1}^{\infty}\Pr\left(\left|\Delta_{k,i}\right|\one_{\left|\Delta_{k,i}\right|\ge M\sqrt{n}}\ge t\right)
\\
 & =M\sqrt{n}\Pr\left(\left|\Delta_{k,i}\right|\ge M\sqrt{n}\right)+\sum_{t=M\sqrt{n}}^{\infty}\Pr\left(\left|\Delta_{k,i}\right|\ge t\right)\\
 & =M\sqrt{n}e^{-\Omega\left(M^{2}\right)}+\sum_{t=M\sqrt{n}}^{\infty}\exp\left(-\Omega\left(t^{2}/n\right)\right)=e^{-\Omega\left(M^{2}\right)}\sqrt{n},
\end{align*}
uniformly over $M$. The desired result follows for sufficiently large
$M$, by linearity of expectation and Markov's inequality.

Now we prove the first part. For each $k,i$, and each $\left\{ \boldsymbol{x},\boldsymbol{y}\right\} \in\binom{X}{2}$,
the random variable $d_{U_{k,i}}\left(\boldsymbol{x}\right)-d_{U_{k,i}}\left(\boldsymbol{y}\right)$
is of $\left(\left|N_{U_{0}}\left(\boldsymbol{x}\right)\triangle N_{U_{0}}\left(\boldsymbol{y}\right)\right|,1/2\right)$-Littlewood\textendash Offord
type. So, we have $\Pr\left(d_{U_{k,i}}\left(\boldsymbol{x}\right)=d_{U_{k,i}}\left(\boldsymbol{y}\right)\right)=O\left(1/\sqrt{n}\right)$.
Let $H_{k,i}$ be the graph of pairs $\left\{ \boldsymbol{x},\boldsymbol{y}\right\} \in\binom{X}{2}$
satisfying $d_{U_{k,i}}\left(\boldsymbol{x}\right)=d_{U_{k,i}}\left(\boldsymbol{y}\right)$,
so we have $\E e\left(H_{k,i}\right)=O\left(\sqrt{n}\right)$. By
Markov's inequality, with probability at least $1-\beta/\left(400M\right)$
we have $e\left(H_{k,i}\right)=O\left(\sqrt{n}\right)$, in which
case by Tur\'an's theorem $H_{k,i}$ has an independent set $Y_{k,i}$
of size $2\gamma\sqrt{n}$, for some constant $\gamma=\gamma(\beta,M)>0$. The expected proportion
of values of $i$ for which this fails to occur is $\beta/\left(400M\right)$,
and by Markov's inequality again, with probability at least 0.995
it fails for only a $\beta/\left(2M\right)$ proportion.

Also, for each $\boldsymbol{x}\in X$, we have $\E d_{U}\left(\boldsymbol{x}\right)=d/2+O\left(\sqrt{n}\right)$
and $\Var d_{U}\left(\boldsymbol{x}\right)=O\left(n\right)$, so by
Chebyshev's inequality, for sufficiently large $Q$ we have $\left|d_{U}\left(\boldsymbol{x}\right)-d/2\right|\le Q\sqrt{n}$
with probability at least $1-\gamma/200$. By Markov's inequality,
with probability at least 0.995 there is a set $Y$ with at least
$\left(1-\gamma\right)\left|X\right|$ elements of $X$ satisfying
$d_{U}\left(\boldsymbol{x}\right)=d/2+O\left(\sqrt{n}\right)$. For sufficiently large $Q$, this means that for each $x\in X$, $d_{U}\left(\boldsymbol{x}\right)$ lies in the interval between $d/2-Q\sqrt n$ and $d/2+Q\sqrt n$.

With probability at least 0.99 both the above events occur, and we
can take $X_{k,i}=Y_{k,i}\cap Y$ for a $\left(1-\beta/\left(2M\right)\right)$
proportion of possibilities of $i$. This proves the first part of
the lemma.

Finally we prove the second part. Note that $e_{k,0}$ is a translation of $\sum_{u\in U}d_{Z_{k,0}}\left(u\right)$,
which is of $\left(O\left(n\right),1/2\right)$-Littlewood\textendash Offord
type, with all coefficients $O\left(\sqrt{n}\right)$. So, $\Var e_{k,0}=O\left(n^{2}\right)$
and the desired result follows from Chebyshev's inequality for sufficiently
large $Q$ (note that enlarging $Q$ cannot make the first part fail to hold).

\end{proof}

\section{Concluding remarks}

In this paper we proved that for any fixed $C$, if $G$ is an $n$-vertex
graph with no homogeneous subgraph on $C\log n$ vertices, then $G$
induces subgraphs of $\Omega\left(n^{2}\right)$ different sizes. This is best possible, but there are a number of other related questions one could ask about Ramsey graphs. For example, as proposed to us by Tuan Tran, we could ask for $\Omega\left(n^{3}\right)$ induced subgraphs with different numbers of triangles. The methods in this paper might be helpful for this question, but the main obstacle seems to be that one would want a fairly strong anticoncentration inequality for quadratic polynomials in place of \cref{lem:ELO}.

Actually, we think it would be interesting in general to explore the extent to which anticoncentration phenomena occur in random subsets of Ramsey graphs. For example, consider the following problem. Let $A$ be the adjacency matrix of an $O(1)$-Ramsey graph $G$ and let $x\in \{0,1\}^n$ be a uniformly random 0-1 vector, so that $x^T A x$ is the number of edges in a uniformly random induced subgraph of $G$. Is it true that $\Pr(x^T A x=c)=O(1/n)$ for all $c\in \mathbb Z$? This is closely related to a conjecture of Costello (\cite[Conjecture 3]{Cos13}), essentially characterising the matrices $A$ for which this approximately holds.

Another interesting further direction
of research would be to consider the situation where larger homogeneous
subgraphs are forbidden (see \cite{AB89,AB07,AKS03,NT17} for some
examples of theorems of this type). In particular, a natural weakening of an ambitious conjecture of Alon,
Krivelevich and Sudakov \cite{AKS03} is that if $G$ is an $n$-vertex
graph with no homogeneous subgraph on $n/4$ vertices, then this is
already enough for $G$ to induce subgraphs of $\Omega\left(e\left(G\right)\right)$
different sizes.

{\bf Acknowledgements.} We would like to thank Tuan Tran for helpful comments. We would also like to thank Mantas Baksys and Xuanang Chen for carefully reading the paper and finding an oversight in the proof (related to the definition of richness in \cref{sec:tools}).

\end{document}